\documentclass[australian,english]{article}
\usepackage{geometry}
\usepackage{color}
\usepackage{babel}
\usepackage{float}
\usepackage{amsmath}
\usepackage{amsthm}
\usepackage{amssymb}
\usepackage{graphicx}
\usepackage[unicode=true,pdfusetitle,
 bookmarks=true,bookmarksnumbered=false,bookmarksopen=false,
 breaklinks=false,pdfborder={0 0 1},backref=false,colorlinks=true]
 {hyperref}
\hypersetup{
 linkcolor=blue, urlcolor=blue, citecolor=blue, pdfstartview={FitH}, hyperfootnotes=false, unicode=true}

\makeatletter
\theoremstyle{plain}
\newtheorem{thm}{\protect\theoremname}[section]
\theoremstyle{plain}
\newtheorem{prop}[thm]{\protect\propositionname}
\theoremstyle{plain}

\theoremstyle{plain}
\newtheorem{lem}[thm]{\protect\lemmaname}
\theoremstyle{definition}
\newtheorem{defn}[thm]{\protect\definitionname}
\theoremstyle{remark}
\newtheorem*{rem*}{\protect\remarkname}
\theoremstyle{definition}
\newtheorem{example}[thm]{\protect\examplename}
\theoremstyle{remark}
\newtheorem{rem}[thm]{\protect\remarkname}

\newcommand{\RR}{\mathbb{R}}
\newcommand{\CC}{\mathbb{C}}

\usepackage{lmodern}

\usepackage[T1]{fontenc}

\usepackage{ae,aecompl}

\usepackage[T1]{fontenc}

\usepackage{amsmath}

\usepackage{graphicx}

\usepackage{tikz}

\makeatother

\addto\captionsaustralian{\renewcommand{\corollaryname}{Corollary}}
\addto\captionsaustralian{\renewcommand{\definitionname}{Definition}}
\addto\captionsaustralian{\renewcommand{\examplename}{Example}}
\addto\captionsaustralian{\renewcommand{\lemmaname}{Lemma}}
\addto\captionsaustralian{\renewcommand{\propositionname}{Proposition}}
\addto\captionsaustralian{\renewcommand{\remarkname}{Remark}}
\addto\captionsaustralian{\renewcommand{\theoremname}{Theorem}}
\addto\captionsenglish{\renewcommand{\corollaryname}{Corollary}}
\addto\captionsenglish{\renewcommand{\definitionname}{Definition}}
\addto\captionsenglish{\renewcommand{\examplename}{Example}}
\addto\captionsenglish{\renewcommand{\lemmaname}{Lemma}}
\addto\captionsenglish{\renewcommand{\propositionname}{Proposition}}
\addto\captionsenglish{\renewcommand{\remarkname}{Remark}}
\addto\captionsenglish{\renewcommand{\theoremname}{Theorem}}
\providecommand{\corollaryname}{Corollary}
\providecommand{\definitionname}{Definition}
\providecommand{\examplename}{Example}
\providecommand{\lemmaname}{Lemma}
\providecommand{\propositionname}{Proposition}
\providecommand{\remarkname}{Remark}
\providecommand{\theoremname}{Theorem}

\begin{document}
\title{A collection of results relating the geometry of plane domains and the exit time of planar Brownian motion.}

\author{~Maher Boudabra, Andrew Buttigieg, and Greg Markowsky\\
{\small {\tt maher.boudabra@monash.edu} ~~ {\tt ajbut5@student.monash.edu } ~~ {\tt greg.markowsky@monash.edu }}\\
{\small Department of Mathematics,  Monash University, Australia}}

\maketitle

\begin{abstract}
    We prove a number of results relating exit times of planar Brownian with the geometric properties of the domains in question. Included are proofs of the conformal invariance of moduli of rectangles and annuli using Brownian motion; similarly probabilistic proofs of some recent results of Karafyllia on harmonic measure on starlike domains; examples of domains and their complements which are simultaneously large when measured by the moments of exit time of Brownian motion, and examples of domains and their complements which are simultaneously small; and proofs of several identities involving the Cauchy distribution using the optional stopping theorem.
\end{abstract}

\section{Introduction}

In this paper, we prove a series of loosely connected results relating the geometry of domains in the plane with the distribution of the exit time of planar Brownian motion from those domains. Results of this type seem to have originated in such seminal papers as \cite{davis1979} and \cite{BURKHOLDER1977182}, and continue to attract current researchers (see the references contained in \cite{markowsky2020planar} for a large list). We first informally describe our results. \\

In the next section, we show how Brownian motion can be used to prove a staple of complex analysis, which is the conformal invariance of the modulus of a rectangle, or of an annulus. In other words, conformal maps between rectangles which fix the vertices cannot change the ratio of the side lengths, and conformal maps between annuli cannot change the ratio of the radii of the annuli. In the following section, we show how Brownian motion can be used to prove some recent analytic results on harmonic measure on starlike domains due to Karafyllia. Next we discuss the relationship between moments of Brownian exit time and the Hardy norm of domains. We discuss how these quantities can indicate the size of a domain at infinity, but also discuss potentially surprising examples of domains and their complements which are simultaneously large when measured by the moments of exit time of Brownian motion, and examples where domains and their complements are simultaneously small. In the final section, we show how the optional stopping theorem can be used to prove a number of results on the Cauchy distribution which are proved in the literature by other methods. \\

We now fix notation. $B_t$ will always denote a planar Brownian motion, and given a domain $U$ let $\tau(U)$ denote the exit time of planar Brownian motion from $U$; that is, $\tau(U) = \inf\{ t \geq 0: B_t \notin U\}$. The next theorem is the core of planar Brownian motion theory.

\begin{thm} \label{Levy}

If $f$ is a non constant analytic function and $B_{t}$ is a planar Brownian motion starting at a then $f(B_{t})$ is a time-changed Brownian motion starting at $f(a)$, where the time change rate is governed by $$\sigma(t):=\int_{0}^{t}\lvert f'(B_{s})\rvert^{2}ds.$$
In other words, the process $f(B_{\sigma^{-1}(t)})$ is a planar Brownian motion. 

\end{thm}

Lévy's Theorem is often referred to as the {\it conformal invariance principle}, although a more accurate name might be the {\it analytic invariance principle}, since injectivity plays no role in the theorem. This result will play the major role in everything that is to follow.

\section{Exit distributions of Brownian motion from rectangles and annuli, and applications to conformal maps}

\subsection{Applications of $\psi$ and probabilistic proofs for known results.}


We will use the following notation for a rectangle (with $a,b>0$).

\[
R_{a,b} :=(-a,a)\times(-b,b).
\]

Fix a rectangle $R_{a,b}$, and label the vertices clockwise as $z_1, \ldots, z_4$, starting with $z_1= a+bi$. Now suppose that we have a similarly defined rectangle $R_{a',b'}$ with vertices $z'_1, \ldots, z'_4$, and that there is a conformal map $f$ between them such that $f(z_j) = z'_j$ for $j=1, \ldots, 4$; note that in making this last statement we are implicitly using Caratheodory's Theorem (see \cite{krantz2006geometric}), which implies that such a conformal map between Jordan domains extends to a homeomorphism between the boundaries. The following is a standard result in complex analysis.

\begin{thm}
\label{thm:rectangle-eqiuiv} Given the setup described above, we have $\frac{a}{b}=\frac{a'}{b'}$, and furthermore $f(z) = cz$, where $c$ is a positive constant. 
\end{thm}

There are at least two known proofs of this. The standard one uses the concept of extremal length (see \cite{ahlfors2010conformal}, \cite{gardiner2000quasiconformal}, or \cite{keen2007hyperbolic}). The second, which is not as widely known, is to repeatedly apply Schwarz's reflection principle in order to extend $f$ to a conformal self-map of $\CC$. We present here a third, using the conformal invariance of Brownian motion.

\begin{proof}
Label the right side of $R_{a,b}$ as $S_1$, and then label the remaining sides $S_2, S_3, S_4$ clockwise. Label the sides of $R_{a',b'}$ similarly as $S'_1, \ldots, S'_4$. Let 

$$
W = \{z \in R_{a,b}: P_z(B_{\tau(R_{a,b})} \in S_1) = P_z(B_{\tau(R_{a,b})} \in S_3)\}.
$$

It is clear by symmetry that $i\RR \cap R_{a,b} \subseteq W$, where $i\RR$ denotes the imaginary axis. It is also not hard to see that $W \subseteq i\RR \cap R_{a,b}$, since a Brownian motion starting from some point not on $i\RR \cap R_{a,b}$ has a positive probability to exit on the nearest vertical side before hitting $i\RR \cap R_{a,b}$, but if it hits $i\RR \cap R_{a,b}$ before hitting the nearest vertical side it has equal probabilities of exiting on $S_1$ and $S_3$ after that by symmetry and the strong Markov property. We conclude that $W = i\RR \cap R_{a,b}$. By the conformal invariance of Brownian motion, for $z \in W$ we must have $P_{f(z)}(B_{\tau(R_{a',b'})} \in S'_1) = P_{f(z)}(B_{\tau(R_{a',b'})} \in S'_3)$, and this implies that $f(i\RR \cap R_{a,b}) = i\RR \cap R_{a',b'}$. The same argument, applied to the horizontal sides, shows that $f(\RR \cap R_{a,b}) = \RR \cap R_{a',b'}$. In particular, we see that $f(0) = 0$. Furthermore, $R_{a,b}$ and $R_{a',b'}$ are divided into four smaller rectangles by the real and imaginary axes, and we have shown that $f$ maps each of these rectangles in $R_{a,b}$ onto the corresponding one in $R_{a',b'}$. Applying the same argument to each of these smaller rectangles, and then iterating the argument shows that $f(\frac{j}{2^n} a + \frac{k}{2^n} bi) = \frac{j}{2^n} a' + \frac{k}{2^n} b'i$, where $n$ is any positive integer and $j,k$ are any odd integers with $|j|, |k| < 2^n$. These points form a dense set in $R_{a,b}$ and $f$ is continuous, so $f$ must be the map $(x,y) \to (\frac{a'}{a}x, \frac{b'}{b}y)$. However, the Cauchy-Riemann equations now imply that $\frac{a'}{a} = \frac{b'}{b}$, and the result follows.
\end{proof}

There is a well-known analogue for an annulus (again, see \cite{ahlfors2010conformal}, \cite{gardiner2000quasiconformal}, or \cite{keen2007hyperbolic}). This states that if there is a conformal map between two annuli then the ratios of their inner and outer radii must be the same, and that the map must be either a linear map or an inversion. Two standard proofs are well-known which parallel the ones for the rectangle, one using extremal length and the other using Schwarz reflection. We will offer here a third using Brownian motion, and will note that our proof yields a more general result which applies to analytic functions which are not necessarily univalent. To state this generalisation, we need the definition of a proper map.

\begin{defn}
\label{proper}
 We say that an analytic function $f:U\rightarrow W$ is proper if for any compact set $K$, $f^{-1}(K)$ is also compact. 
\end{defn}

Definition \ref{proper} extends to any continuous function between two topological spaces \cite{rudin2012function}. The interpretation of properness is the following : If $(z_{n})_{n}$ is a sequence of $U$ that converges to $\partial U$, i.e the complement of any compact set of $U$ contains all but a finite number of $(z_{n})_n$, then $(f(z_{n}))_{n}$ converges necessarily to $\partial W$. Often, we express that as $f$ maps boundary to boundary. Note that conformal maps are automatically proper, since $f^{-1}$ is itself conformal.

We will use the following notation for an annulus 
\[
\begin{alignedat}{1}A_{h,r} & :=\{z\mid r<\vert z\vert<R\}.
\end{alignedat}
\]


Our generalisation is as follows.

\begin{thm}
\label{thm:annulusequiv} Suppose $f:A_{r,R}\rightarrow A_{r',R'}$ is
a proper analytic function. Then there is a positive integer $n$ such that $\frac{R^n}{r^n} = \frac{R'}{r'}$, and $f$ is of the form $f(z)=\xi z^{n}$ or $f(z)= \frac{\xi}{z^{n}}$ for some constant $\xi$.
\end{thm}

\begin{proof}
We will consider a Brownian motion $B_t$ running in $A_{r,R}$, and the image time-changed Brownian motion $B'_t$ running in $A_{r',R'}$. Since $f$ is proper, it is easy to see that $\sigma(\tau(A_{r,R})) = \tau'(A_{r',R'})$, where $\tau'(A_{r',R'})$ denotes the exit time of the Brownian motion $B'_t = f(B_{\sigma^{-1}(t)})$. Fix $a \in A_{r,R}$ and choose $r'',R''$ such that $r'<r''< |f(a)| < R'' < R'$. Since $f$ is proper, $f^{-1}(cl(A_{r'',R''}))$ is compact (here $cl$ denotes the closure), and we may therefore find $\tilde r>r$ and $\tilde R < R$ such that $f^{-1}(cl(A_{r'',R''})) \subseteq A_{\tilde r, \tilde R}$. The sets $f(\{|z|= \tilde r\})$ and $f(\{|z|= \tilde R\})$, must lie in the complement of $A_{r'',R''}$, however their union must separate $f(a)$ from the boundary of $A_{r',R'}$; to see this, note that the set $\{B_t: 0 \leq t \leq \tau(A_{r,R})\}$ must intersect $\{|z|= \tilde r\}) \cup \{|z|= \tilde R\}$, and thus  $\{B'_t: 0 \leq t \leq \tau'(A_{r',R'})\}$ must intersect $f(\{|z|= \tilde r\}) \cup f(\{|z|= \tilde R\})$. By continuity, each of $f(\{|z|= \tilde r\})$ and $f(\{|z|= \tilde R\})$ must lie in just one of $\{|z| < r''\}$ and $\{|z|>R''\}$. We can then, if necessary, replace $f$ by $\frac{Rr}{f}$ and assume that $f(\{|z|= \tilde r\}) \subseteq \{|z| < r''\}$ and $f(\{|z|= \tilde R\}) \subseteq \{|z|>R''\}$. Now let $\tilde r \searrow r$ and $\tilde R \nearrow R'$; as $f$ is proper, the distances from the images $f(\{|z|= \tilde r\})$ and $f(\{|z|= \tilde R\})$ to the boundary of $A_{r',R'}$ must go to 0, and we conclude the following equality of events: 

\begin{gather*}
\{|B_{\tau(A_{r,R})}| = r\} = \{|B'_{\tau'(A_{r',R'})}| = r'\} \\
\{|B_{\tau(A_{r,R})}| = R\} = \{|B'_{\tau'(A_{r',R'})}| = R'\}
\end{gather*}

Fix $\theta \in (0,1)$, and let $W = \{a: P_a(|B_{\tau(A_{r,R})}| = r) = \theta\}$. It is well known that $W$ is the circle $\{\ln |z| = \ln R - \theta \ln \frac{R}{r}\}$; see for example \cite[Thm. 3.17]{morters2010brownian}. By conformal invariance, we must have $f(W) \subseteq W'$, where $W' = \{a: P_a(|B'_{\tau'(A_{r',R'})}| = r')=\theta\} = \{\ln |z| = \ln R' - \theta \ln \frac{R'}{r'}\}$. Rearranging and dividing yields

\[
\frac{\ln|f(z)/R'|}{\ln|z/R|} = \frac{\ln|R'/r|}{\ln|R/r|}.
\]

Setting $\eta:=\frac{\ln|R'/r|}{\ln|R/r|}$, we obtain 
\[
|f(z)|=\frac{R'}{R^{\eta}}|z|^{\eta}.
\]

\if2
But since 
\[
\begin{alignedat}{1}\frac{\rho}{r^{\eta}} & =\frac{\rho}{r^{\frac{\ln(\rho)}{\ln(r)}}}\\
 & =\rho\exp(\ln(r^{-\frac{\ln(\rho)}{\ln(r)}}))\\
 & =\rho\exp(\ln(\frac{1}{\rho}))\\
 & =1,
\end{alignedat}
\]

we have

\[
|f(z)|=\frac{\rho}{r^{\eta}}|z|^{\eta}=|z|^{\eta}.
\]
\fi 

We want to now consider $\frac{f(z)}{z^\eta}$ as an analytic function, however it can not be defined to be analytic on the entire annulus unless $\eta$ is an integer. Since we do not, as of yet, know that it is an integer, we solve this problem by removing the slit $S=\{Im(z) = 0, -R<Re(z)<-r\}$, from the annulus, and choose a branch of the function $z \to z^\eta$ which is analytic on $A_{r,R} \backslash S$. On this domain, the function $\frac{f(z)}{z^\eta}$ is indeed analytic, and as it has constant absolute value it must be a constant by the open mapping theorem for analytic functions. Thus, $f(z)=\xi z^{\eta}$ for some constant $\xi$. However, $f$ is analytic on all of $A_{r,R}$, and the result now follows upon recalling again that $z \to z^{\eta}$ can only be analytic on the annulus if $\eta$ is an integer $n$.
\end{proof}

\begin{rem*}
The ratio of the two dimensions of a rectangle (resp. inner and outer radii of an annulus) is referred to as the {\it modulus} of the rectangle (resp. annulus), and these results imply that it is a conformal invariant. Interestingly, a different probabilistic interpretation of modulus has recently been given in \cite{albin2021convergence}. The approach taken there is quite different from ours, and makes use of discretization. 
\end{rem*}

\section{Brownian Motion and Starlike Domains }

\subsection{Exit of Brownian motion from a Starlike Domain }

The main purpose of this section is to examine a recent result of Karafyllia \cite{karafyllia_2019} on the harmonic measure of starlike domains. We will show how the result can be extended, and will give a proof using Brownian motion. First, we need some definitions. 

\begin{defn}
(Harmonic Measure)

For a simply connected domain $D$ and a subset $A\subset\partial D$,
we define the harmonic measure as 
\[
\omega_D(a,A):=\mathbb{P}_{a}(B_{\tau(D)}\in A)
\]
\end{defn}

{\it Remark:} The harmonic measure can also be defined analytically, i.e. without mentioning Brownian motion. This is the definition used by Karafyllia in \cite{karafyllia_2019}, but the two definitions are equivalent.

\begin{defn}
(\foreignlanguage{australian}{Starlike} Domain)

An open set $D$ on $\mathbb{C}$ is starlike (with respect to $z_0$ if there exists some
$z_{0}\in D:\{(1-t)z_{0}+tz:t\in[0,1]\}\in D$ for all $z\in D$. 
\end{defn}

\begin{defn}
($\triangle-$Starlike Domain)

A domain $U$ is $\triangle-$starlike if, given any $z\in U$,
the horizontal ray 
\[
\triangle(z)=\{w:\Im(w)=\Im(z),\Re(w)\le\Re(z)\}
\]
 lies entirely in $U$.

\begin{figure}[H]
\begin{centering}
\includegraphics[scale=0.2]{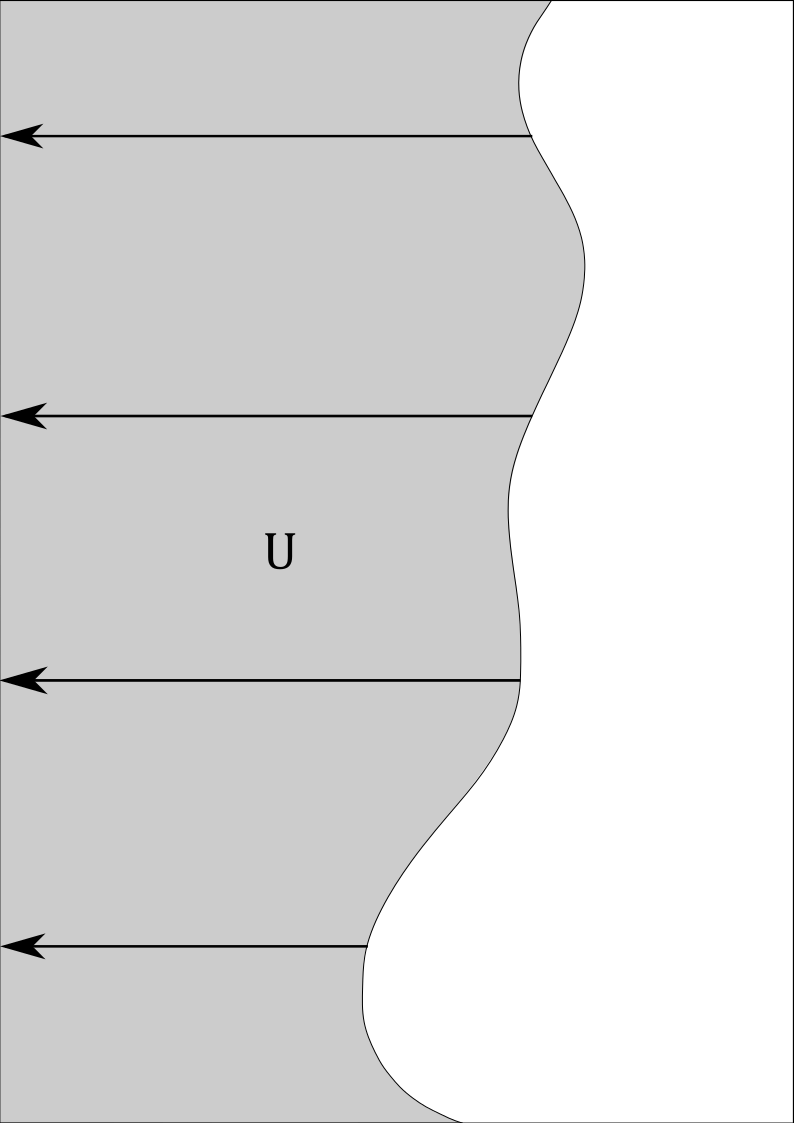}
\par\end{centering}
\caption{All rays going left lie inside U}
\end{figure}
\end{defn}

The following result was proved analytically for starlike domains
by Karafyllia \cite{karafyllia_2019}, in order to answer a question posed by Betsakos in \cite{betsakos2001geometric}. 

\begin{thm}
Let $D$ be a domain in $\mathbb{C}$ which is starlike with respect to $0$. Then for every $R>0$,

\[
\hat{\nu}_{D}(0,R)\le2\nu_{0,D}(R),
\]

where $\nu_{D}(0,R):=\omega_D(0,\partial D\cap\{|z| > R\})$, $\hat{\nu}_{D}(0,r):=\omega_{D\cap \{|z|<R\}}(0,D\cap\{|z|=R\})$.
The constant $2$ is the best possible.
\end{thm}

Now a generalisation
is offered and proved using the reflection principle of Brownian motion. Our result pertains to $\Delta$-starlike domains, and to see that it really is a generalisation of Karafyllia's result we note that any domain starlike with respect to 0 corresponds to a $\triangle-$starlike domain with periodic boundary of period $2\pi i$ through the complex exponential map. The image of this map will be the starlike domain with the origin removed, but the removal of the origin is irrelevant to the Brownian motion since planar Brownian motion doesn't see points (see \cite{markowsky2020planar}), provided that the Brownian motion doesn't start at the origin. To include the origin, simply take a point near the origin, conclude the theorem for that starting point, and then let the point approach the origin. To summarize, our result can be translated to Karafyllia's by the exponential map. The following is our result.

\begin{thm}
Let U be a $\triangle-$starlike Domain in $\mathbb{C}$.
Then for every $r\in\mathbb{R},$ $a\in \{\Re(z) < r\}$, we have

\[
\hat{\nu}_{U}(a,r)\le2\nu_{U}(a,r),
\]

where $\nu_{U}(a,r):=\omega_U(a,\partial U\cap\{\Re(z)> r\})$ and
$\hat{\nu}_{U}(a,r):=\omega_{U \cap \{\Re(z)<r\}}(a,U\cap\{\Re(z)=r\})$. The constant
2 is the best possible.
\end{thm}

{\it Remark:} In addition to transferring the result to $\Delta$-starlike domains, we also show that the harmonic measure can be taken with respect to any point in $a \in \{|z|< R\}$, rather than just $0$. That is, we prove

\[
\hat{\nu}_{D}(a,R)\le2\nu_{a,D}(R),
\]

for any $a \in \{|z|< R\}$. In terms of the Brownian motion interpretation of harmonic measure, this means that the Brownian motion can start at any $a \in \{|z|< R\}$ and the result still holds.
\begin{proof}
It clearly suffices to assume $r=0$. Let $\mathbb{I}:=\{z:\Re(z)=0\}$, $\tau(U):=\inf\{t:B_{t}\notin U\}$,
$\tau_{0}:=\inf\{t:B_{t}\notin\{\Re(B_{t})<0\}\cap U\}$. Define the
coupling 

\[
\hat{B}_{t}=\begin{cases}
B_{t} & t\le\tau_{0}\\
-\bar{B}_{t} & t>\tau_{0}
\end{cases}.
\]

By the reflection principle for planar Brownian motion, $\hat{B}_{t}$
is a Brownian motion. It is evident that, as a consequence
of $\triangle-$starlikeness,
\[
\{\Re(\hat B_{\tau(U)})<0, \tau(U)>\tau_0\}\subseteq\{\Re(B_{\tau(U)})>0, \tau(U)>\tau_0\}.
\]

Figure \ref{ref_pic} captures this idea, that once the Brownian motion hits $\mathbb{I}$ it will be at least as likely to exit to the right of this line than to the left.

\begin{figure}[H]
\begin{centering}
\includegraphics[scale=0.32]{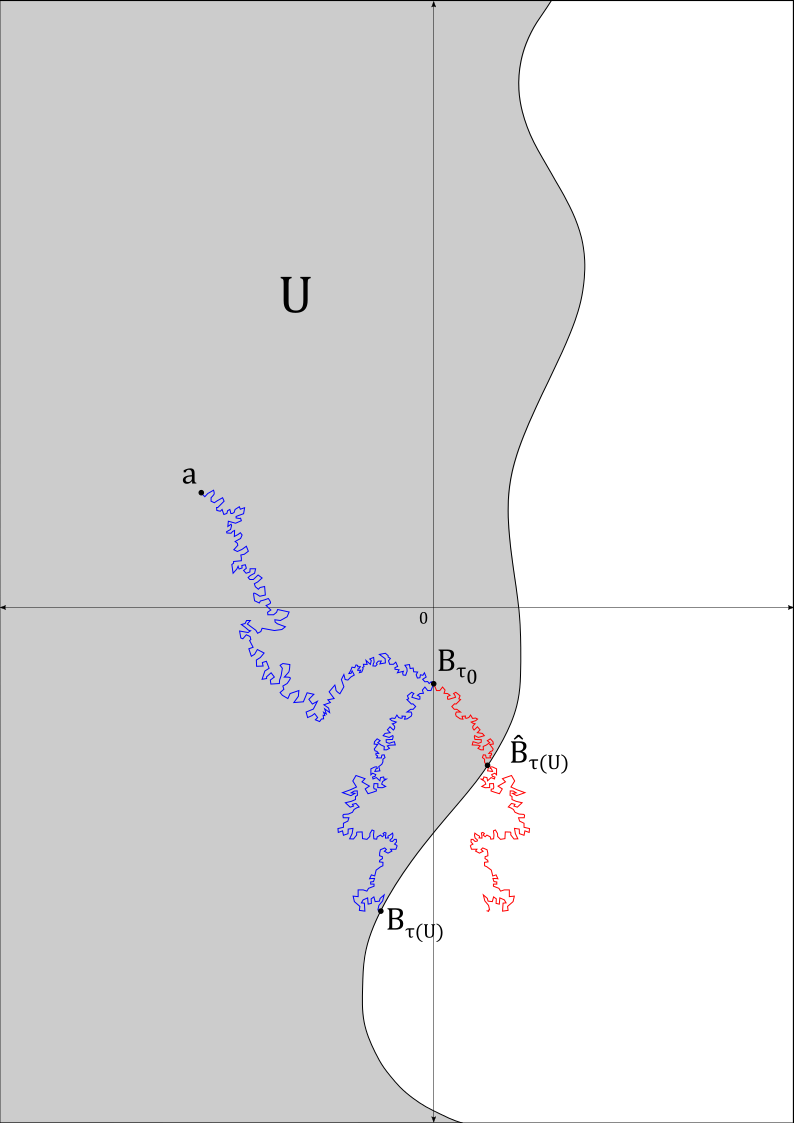}\includegraphics[scale=0.32]{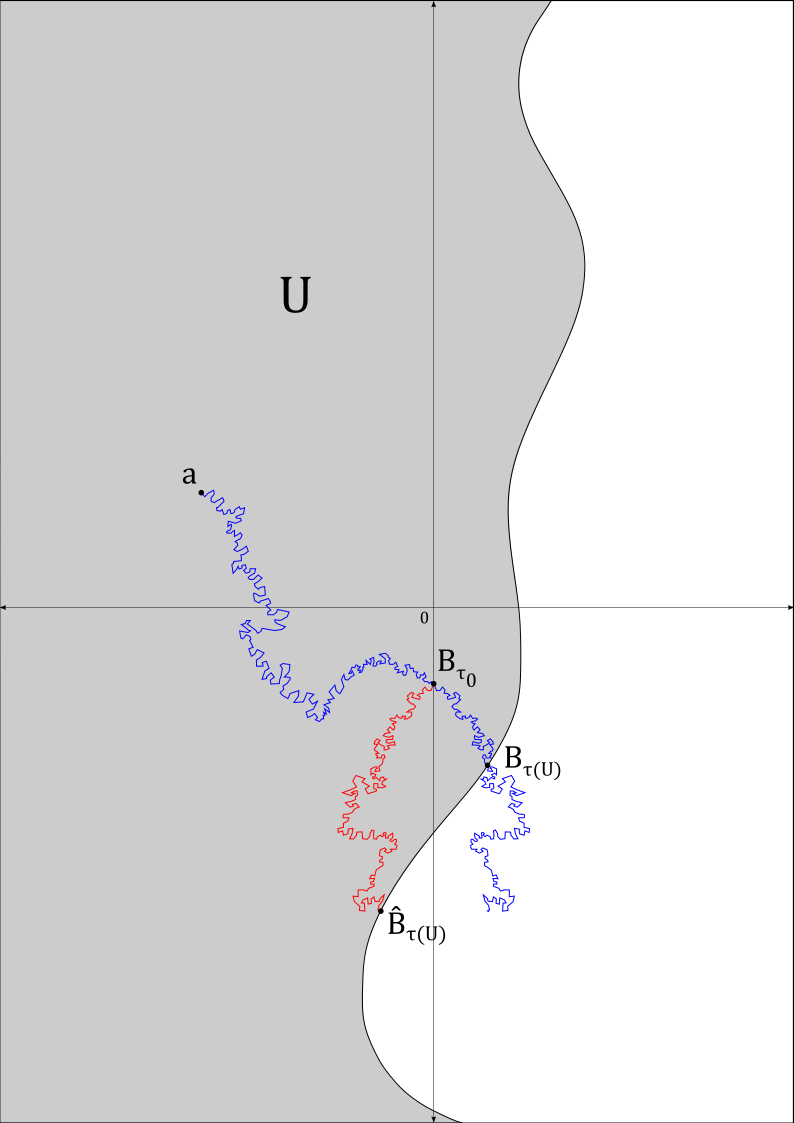}
\par\end{centering}
\begin{centering}
\caption{On the left, if $B_{\tau(U)}$ is left of $\mathbb{I}$, $\hat{B}_{\tau(U)}$
is guaranteed to lie on the right of $\mathbb{I}$. On the other hand if $B_{\tau(U)}$ lies to the right
of $\mathbb{I}$, then $\hat{B}_{t}$ may have lived past that time, and may or may not exit to the left of $\mathbb{I}$. } \label{ref_pic}
\par\end{centering}
\end{figure}

Note also that, up to a set of probability 0, $\{\tau(U)> \tau_0\}$ is the disjoint union of $\{\Re(B_{\tau(U)})<0, \tau(U)>\tau_0\}$ and $\{\Re(B_{\tau(U)})>0, \tau(U)>\tau_0\}$. We see that

\[
P(\Re(B_{\tau(U)})>0, \tau(U)>\tau_0) \geq \frac{1}{2}P(\tau(U)> \tau_0).
\]

The result now follows upon noting that 

$$P(\Re(B_{\tau(U)})>0, \tau(U)>\tau_0)= P(\Re(B_{\tau(U)})>0) = \nu_{U}(a,0)$$ 

and 
$$P(\tau(U)> \tau_0) = \hat \nu_{U}(a,0)$$

\end{proof}
\begin{rem*}

As an alternative to the proof given here, one may obtain Karafyllia's result directly by reflecting over the circle of radius $R$. Such a reflection does preserve Brownian motion but also introduces a time change, however that is not relevant for our problem since we are only interested in the distribution of the Brownian motion at time $\tau$, rather than the distribution of $\tau$ itself. The constant $2$ in our result is optimal, and occurs when the $\triangle-$starlike
domain is a horizontal strip or half-plane with horizontal boundary. We note that the constant 2 is also optimal in Karafyllia's result as well, but in that case the inequality is always strict; that is, the constant 2 is never obtained for a given domain. This is because the starlike domain can never actually equal its reflection over the circle, since $0$ is in the starlike domain but $\infty$ is in its complement.
\end{rem*}

\section{Domains with constraints on the exit moments}

\subsection{Hardy Spaces and $p$-th moments of exit times}

The theory of Hardy spaces provides another point of intersection
between the theory of Brownian motion and complex analysis. Hardy
spaces provide an analytic way of putting a measure on the size of
the domain as seen by Brownian motion through its exit time out of
the domain. Moreover, Hardy spaces play a major role when it comes to the problem of finiteness of the moments of the exit times. To this end, let $f$ be a analytic function on the unit disc and for any $p>0$ and $0\leq r<1$ set 

$$N_{p,r}(f):=\left\{ \frac{1}{2\pi}\int_{0}^{2\pi}|f(re^{i \theta})|^{p}d\theta\right\} ^{\frac{1}{p}}.$$
The quantity $N_{p,r}(f)$ can be interpreted as the $L^{p}$ norm \footnote{The word norm is an abuse of language as $N_{p,r}(f)$ is not a true norm when $p<1$.} of the function $\theta\mapsto f(re^{\theta i})$. It can be shown, using harmonic analysis techniques, that $N_{p,r}(f)$ is non decreasing in terms of $r$ \cite{rudin2012function}. Hence, we are ready now to define the $p^{th}$-Hardy norm.

\begin{defn}
For any analytic function $f$ on the unit disc, the $p^{th}$-Hardy norm of $f$ is defined by $$H_{p}(f):=\sup_{0\leq r<1}N_{p,r}(f)=\sup_{0\leq r<1}\left\{ \frac{1}{2\pi}\int_{0}^{2\pi}|f(re^{\theta i})|^{p}d\theta\right\} ^{\frac{1}{p}}.$$
\end{defn}
The set of analytic functions whose $p^{th}$-Hardy norm is finite is denoted by $\mathcal{H}^{p}$ and called Hardy space (of index $p$). A crucial result about Hardy norms is that, if $H_{p}(f)$ is finite then $f$ has a radial extension to the boundary. More precisely $f^{\ast}(z):=\lim_{r\rightarrow1}f(rz)$ exists a.e for all $z\in\partial\mathbb{D}$ and it belongs to $L^{p}$ as well. More details about the topic can be found in \cite{duren2000theory}. A consequence of Hölder's inequality is the inclusion $\mathcal{H}^{q}\subseteq\mathcal{H}^{p}$ whenever $0<p\leq q$. This leads to the following definition.

\begin{defn}
(Hardy Number)

Given a simply connected domain $W$ and conformal $f:\mathbb{D\rightarrow}W$,
the Hardy number, $H(W)$, of $W$ is defined as 

\[
H(W):=\sup\{p>0 \mid  H_{p}(f)<\infty\}.
\]
\end{defn}

Although the definition is in terms of $f$ but it can be shown that the Hardy number does not depend of the functional mapping $\mathbb{D}$ onto $U$ $f$ and so it is well defined. We will be interested in the Hardy numbers of a particular type of domain, namely Jordan* domains, which we now define.

\begin{defn}
(Jordan curve)

A Jordan curve is a closed curve in $\CC$ which is the homeomorphic image of a circle.
\end{defn}

\begin{defn}
(Jordan domain)

A domain whose boundary is a Jordan curve is called a Jordan domain.
\end{defn}

Note that Jordan domains are automatically simply connected. The question of calculating the Hardy number of a Jordan domain is not interesting, because a Jordan domain is bounded and it follows easily that its Hardy number is infinite. In order to make the question more interesting, we describe another type of domain, the Jordan{*} domains, which were previously considered in \cite{markowsky2015exit}.

\begin{defn}
(Jordan{*} Domain)

A Jordan{*} domain is a domain in $\CC$ which is the image of Jordan domain $U$ under a Mobius transformation which takes a boundary point of $U$ to $\infty$.
\end{defn}

Essentially we are modifying the class of Jordan domains to require $\infty$ to be a boundary point, and the Hardy number of Jordan* domains is no longer trivial. Alternatively, Jordan* domains could be defined to be domains whose boundary contains $\infty$ and is homeomorphic (as a set in the Riemann sphere) to a circle. An important paper which relates the theory of Hardy spaces to the exit time of Brownian motion is that of Burkholder
\cite{BURKHOLDER1977182} and we will implement from it without proof
the following two results.
\begin{thm}
\label{thm:hardyexitequiv}Suppose $f$ is a conformal map taking
the unit disc to a simply connected domain $R$. Then for all $x\in R$,

\[
f\in H^{2p}\iff\mathbb{E}_{x}[\tau(R)^{p}]<\infty\iff\mathbb{E}_{x}[|B_{\tau(R)}|^{2p}]<\infty.
\]
\end{thm}

The $p$-Hardy norm can be seen as the expectation of the $p^{th}$
moment of $f(B_{T(\mathbb{D})})$ as $B_{T(\mathbb{D})}$ is uniformly
distributed on the circle. Thus, the equivalence of the finiteness
of the $p^{th}$ moment of the stopped Brownian motion and the $2p^{th}$
Hardy norm of a conformal map can be seen as a consequence of the
conformal invariance of Brownian motion. Another crucial fact about Theorem \ref{thm:hardyexitequiv} is that the finiteness of $\mathbb{E}_{x}[\tau(R)^{p}]$ and $\mathbb{E}_{x}[|B_{\tau(R)}|^{2p}]$ does not depend on the starting point $x$ \cite{BURKHOLDER1977182}. Burkholder was also able to give
the following general result about the $p^{th}$ moment of the exit time of a
Brownian motion exiting any simply connected domain, which is a consequence of the corresponding result for Hardy norms (see \cite{duren2000theory}).

\begin{thm}
\label{thm:simplyconnecthardy}If $R$ is simply connected and not
the whole complex plane then, for all $z\in R$, $\mathbb{E}_{z}[\tau(R)^{p}]<\infty$
for any $p<\frac{1}{4}$.
\end{thm}

In other words, the minimum Hardy number for any simply connected domain
is $\frac{1}{2}$. This is therefore considered an extremal case with
regards to the size of a simply connected domain as seen by Brownian
motion. The Koebe domain, $K:=\mathbb{C}\backslash(-\infty,-\frac{1}{4}]$
is considered in many ways extremal among simply connected domains,
and it has infinite $1/4^{th}$ moment of $\tau_{K}$
and thus a Hardy number of $\frac{1}{2}$ (this is proved in \cite{duren2000theory}, and can also be proved using Burkholder's result and calculations similar to Example \ref{exa:halfstriphardy} below). 
There are examples of smaller domains which also have infinite $1/4^{th}$
moment. One of these will be important in what follows, so we take the time to describe it.

\begin{example}
\label{exa:halfstriphardy}Consider a half-strip, for example $\{\Re(z)<0,-1<\Im(z)<1\}$
and label the complement of this half strip $V$. Now consider a slightly
smaller domain $W$, which is the complement of the parabola given by the equation
$x=1-\frac{y^{2}}{4}$. Also consider the conformal map from
the unit disk to $W$, $f(z)=\frac{4}{(1+z)^{2}}.$ We see

\[
\begin{alignedat}{1}H_{p}(f)^p & ={\textstyle \frac{1}{2\pi}}\int_{0}^{2\pi}\vert f(e^{ti})\vert^{2p}dt\\
 & ={\textstyle \frac{1}{2\pi}}\int_{0}^{2\pi}\Bigl|\frac{4}{(1+e^{it})^{2}}\Bigr|^{2p}dt.
\end{alignedat}
\]

The singularity at $t=\pi$ is only contained when $p<\frac{1}{4}$,
thus the $2p^{th}$-Hardy norm is infinite for $p=1/4$. By Theorem \ref{thm:hardyexitequiv} we have
$f\notin H^{\frac{1}{2}}$, meaning $\mathbb{E}[\tau_{W}^{\frac{1}{4}}]=\infty$
and since $W\subset V$, $\mathbb{E}[\tau_{V}^{\frac{1}{4}}]=\infty$
hence $H(V)=\frac{1}{2}$. We may translate the parabola as required
such that any half strip may be fit inside, meaning any half strip
will have infinite $1/4^{th}$ moment of its exit time out of the
domain and thus a Hardy number of $\frac{1}{2}$. 
\end{example}

As another application of Theorem \ref{thm:hardyexitequiv} (applied to the conformal map $f(z)=(\frac{1-z}{1+z})^{\frac{\theta}{\pi}}$), let us consider the following result (which can also be deduced from the explicit formula for the distribution of $\tau(W)$ given in \cite{Markowsky2018}).

\begin{prop} \cite{BURKHOLDER1977182}
\label{prop:wedgehardy}Let $W=\{-\frac{\theta}{2}<Arg(z)<\frac{\theta}{2}\}$ be
a wedge of aperture $\theta \in (0,2\pi]$. Then

\[
\mathbb{E}(\tau(W)^{\frac{p}{2}})<+\infty\Longleftrightarrow p<\frac{\pi}{\theta}.
\]
\end{prop}

Note that the Koebe domain can be considered as a wedge of size $2\pi$
with $\theta=\pi$ giving the required result of 
\[
\mathbb{E}(\tau^{\frac{1}{4}})<+\infty\Longleftrightarrow p<\frac{1}{2},
\]
resulting in a Hardy number of $\frac{1}{2}$ as expected. \\

We will now give an intuitive discussion to motivate a natural conjecture that we have considered. Proposition \ref{prop:wedgehardy} indicates that the aperture of a domain at $\infty$ essentially decides whether the domain is viewed as "small" or "large" by Brownian motion; we will continue to use these terms in quotation marks to indicate their informal nature. Consider now a wedge $V$ of size $\theta \in (0,2\pi)$
and its complement $W$ of size $\alpha = 2\pi-\theta$. Suppose that $\mathbb{E}(\tau_{V}^{\frac{p}{2}})<+\infty$ for some $p$. This means that $V$ is "small", so $W$ must be "large". To be precise, suppose $\frac{1}{p}+\frac{1}{q} \le 2$; we will show that $\mathbb{E}(\tau_{W}^{\frac{q}{2}})=+\infty$. We know $p < \frac{\pi}{\theta}$, so

\begin{equation*}
        2 > \frac{1}{q} + \frac{\theta}{\pi} = \frac{1}{q} + \frac{2\pi - \alpha}{\pi},
\end{equation*}

hence $q > \frac{\pi}{\alpha}$ and thus $\mathbb{E}(\tau_{W}^{\frac{q}{2}})=+\infty$. On the other hand, if we suppose $\mathbb{E}(\tau_{V}^{\frac{p}{2}})=+\infty$ (i.e. $V$ is "large")
and $\frac{1}{p}+\frac{1}{q}>2$ then we can show $\mathbb{E}(\tau_{W}^{\frac{q}{2}})<+\infty$ (i.e. $W$ is "large"). To see this, note that $p \geq \frac{\pi}{\theta}$, so

\begin{equation*}
        2 < \frac{1}{q} + \frac{\theta}{\pi} = \frac{1}{q} + \frac{2\pi - \alpha}{\pi},
\end{equation*}

hence $q < \frac{\pi}{\alpha}$ and $\mathbb{E}(\tau_{W}^{\frac{q}{2}})<+\infty$. These observations have led us to an attempt at generalization, and the Jordan* domains appear to be ideally suited for this purpose, since the Jordan Curve Theorem implies that the interior of the complement of a Jordan* domain is again a Jordan* domain. The following conjecture is therefore natural.\\

{\bf Conjecture:} If $V$ and $W$ are complementary Jordan* domains (that is, they share the same boundary), then they cannot both be "small" or both be "large" as viewed by Brownian motion. \\

We have not attempted to make this statement more rigorous, since we have found example domains that show that virtually any formulation of it will be false. We will now take the time to describe them; in particular, we will prove the following theorem.

\begin{thm}
\begin{enumerate}
    \item There exist complementary Jordan* domains $V$ and $W$ such that $\mathbb{E}[\tau(V)^{1/4}] = \mathbb{E}[\tau(W)^{1/4}] = \infty$.
    
    \item There exist complementary Jordan* domains $V$ and $W$ such that $\mathbb{E}[\tau(V)^{p}] < \infty$ and $\mathbb{E}[\tau(W)^{p}] < \infty$ for all $p>0$.
\end{enumerate}
\end{thm}

Part {\it 1} of the theorem provides an example of complementary Jordan* domains which are both "large" as viewed by Brownian motion, and part {\it 2} does the same with both "small". We remind the reader that any Jordan* domain is simply connected and therefore has Hardy number at least $1/2$. This implies that part {\it 1} of the theorem cannot be improved. The remainder of this subsection will provide the proof of part {\it 1}, while part {\it 2} will be proved in the next subsection.

\begin{proof}
We will utilise the following result for domains, which is a straightforward consequence of the monotone convergence theorem.

\begin{thm}
\label{thm:combthm}If $\Omega_{n}$ is an increasing sequence of
domains (i.e. $\Omega_{n}\subseteq\Omega_{n+1}$) and $\Omega=\cup_{n=1}^{\infty}\Omega_{n},$
then $\mathbb{E}[\tau(\Omega_{n})^{p}]\rightarrow\mathbb{E}[\tau(\Omega)^{p}].$
\end{thm}

We will construct a sequence of domains $(W_{n})_{n\in\mathbb{N}_{0}}$
and a sequence of each respective complement $(V_{n})_{n\in\mathbb{N}_{0}}$
in which the limiting domains of each sequence will both have infinite
$1/4^{th}$ moment. In each $W_{n}$ we start a Brownian motion $B_{t}$
at $-1$ and in each $V_{n}$ we start a Brownian motion $\hat{B}_{t}$
at 1. 

We begin with an extension of Example \ref{exa:halfstriphardy}, in which we restrict
the complement of the half strip by the lines shown in Figure \ref{fig:1stiterationCE1}
\[
\Re(z)=b_{1},\quad\Im(z)=a_{1}\quad\Im(z)=-a_{1},
\]

where $a_1 > a_0 := 1$
$V_{1}$ is a subset of the complement of a half strip and so, if we
take $a_{1}\rightarrow\infty$ and $b_{1}\rightarrow-\infty$,
then $W_{1}$ becomes a half strip, so $V_{1}$ becomes the complement
of a half strip and by (\ref{thm:combthm}), $\mathbb{E}[\tau(V_{1})^{\frac{1}{4}}]\rightarrow\infty$.
Thus we may take $a_1, b_1$ large enough such that $\mathbb{E}[\tau(V_{1})^{\frac{1}{4}}]>1$, and may also ensure that $a_1 \geq a_0 + 1$. Since the complement of a half strip can fit inside $W_{1}$,
we have that $\mathbb{E}[\tau(W_{1})^{\frac{1}{4}}]=\infty$.

\begin{figure}[H]
\begin{centering}
\includegraphics[scale=0.3]{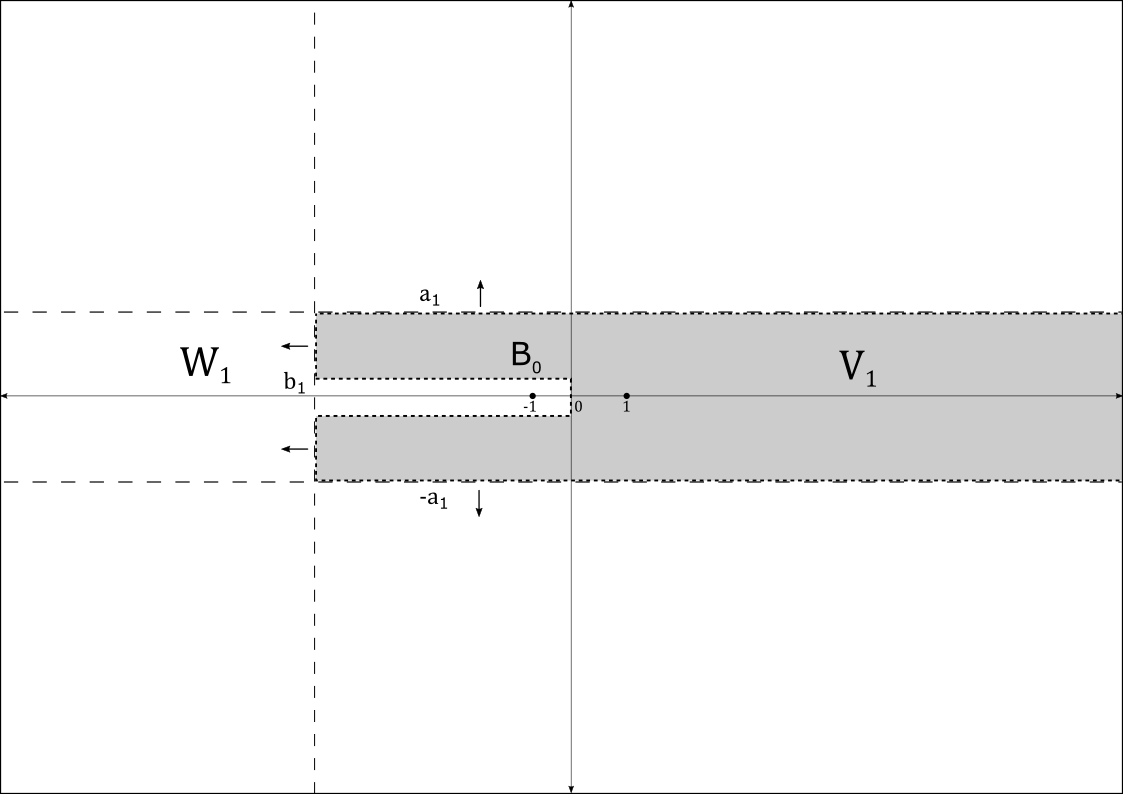}
\par\end{centering}
\caption{\label{fig:1stiterationCE1}$1^{st}$ iteration of domain}
\end{figure}

We can then restrict $W_{1}$ by the lines shown in Figure \ref{fig:2iterationCE1}

\[
\Re(z)=b_{2},\quad\Im(z)=a_{2}\quad\Im(z)=-a_{2}.
\]
As $a_{2}\rightarrow\infty$ and $b_{2}\rightarrow\infty$, 
$W_{2}\rightarrow W_{1}$ and we have that $W_{2}\subset W_{1},$
meaning by Theorem \ref{thm:combthm}, 
\[
\mathbb{E}[\tau(W_{2})^{\frac{1}{4}}]\rightarrow\mathbb{E}[\tau(W_{1})^{\frac{1}{4}}]=\infty.
\]
We may then take these limits large enough such that $\mathbb{E}[\tau(W_{2})^{\frac{1}{4}}]>2$
while again ensuring $a_2 \geq a_1+ 1$. Since the complement of a half strip can fit inside $V_{2}$,
we also have that $\mathbb{E}[\tau(W_{1})^{\frac{1}{4}}]=\infty$. 

\begin{figure}[H]
\begin{centering}
\includegraphics[scale=0.35]{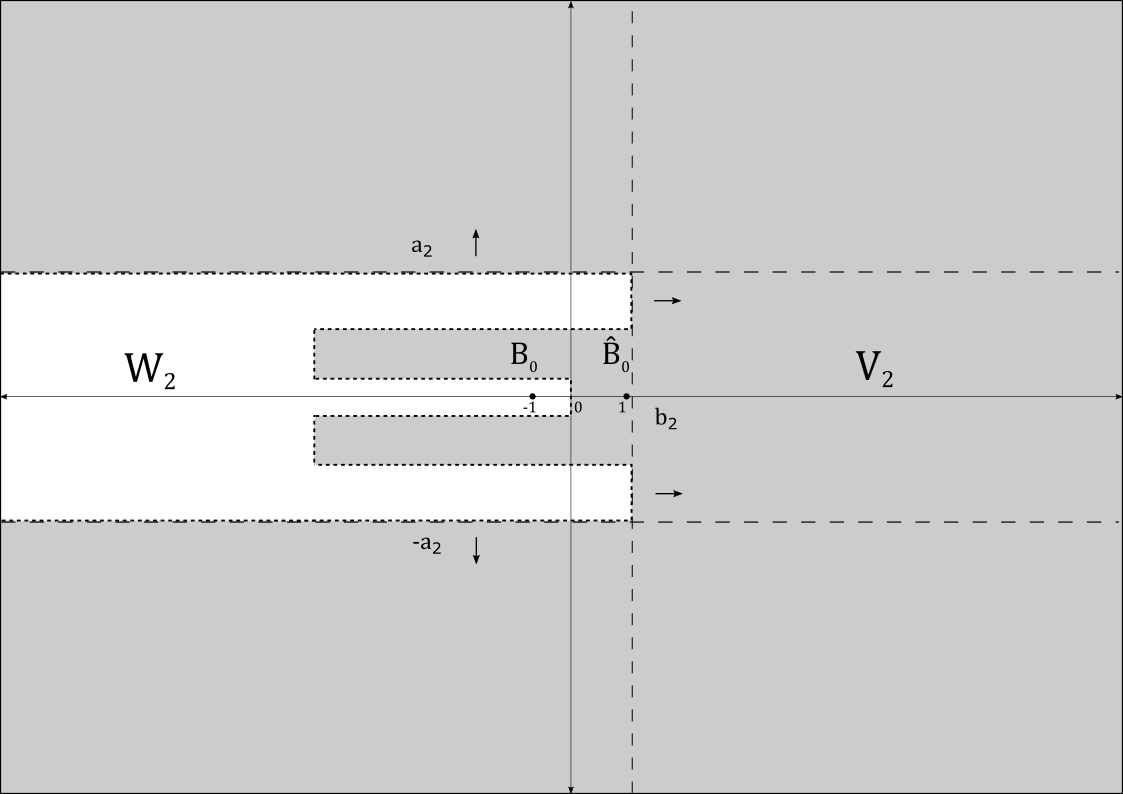}
\par\end{centering}
\caption{\label{fig:2iterationCE1}$2^{nd}$ iteration of domain}

\end{figure}

We may continue iterating this process. In the odd iterations, for
all $n\in\mathbb{N}_{0}$, as $a_{2n+1}\rightarrow\infty$, $b_{2n+1}\rightarrow-\infty$, then $V_{2n+1}\rightarrow V_{2n}$
and since $V_{2n+1}\subset V_{2n}$, by Theorem \ref{thm:combthm}, 
\[
\mathbb{E}[\tau(V_{2n+1})^{\frac{1}{4}}]\rightarrow\mathbb{E}[\tau(V_{2n})^{\frac{1}{4}}]=\infty.
\]
Thus, the limits may be taken large enough such that $\mathbb{E}[\tau(V_{2n+1})^{\frac{1}{4}}]>2n+1$
while once again ensuring that $a_{2n+1} \geq a_{2n}+1$. The complement of a half strip can fit inside $W_{2n+1}$
meaning $\mathbb{E}[\tau(W_{2n+1})^{\frac{1}{4}}]=\infty$. \\

Following the same logic with the as the odd iterations we construct the even
iterations such that $\mathbb{E}[\tau(W_{2n+2})^{\frac{1}{4}}]>2n+2$
and $\mathbb{E}[\tau(V_{2n})^{\frac{1}{4}}]=\infty$, with $a_{2n} \geq a_{2n-1}+1$ . \\

We show three more iterations in Figure \ref{fig:345iterationCE1}.

\begin{figure}[H]
\begin{centering}
\includegraphics[scale=0.2]{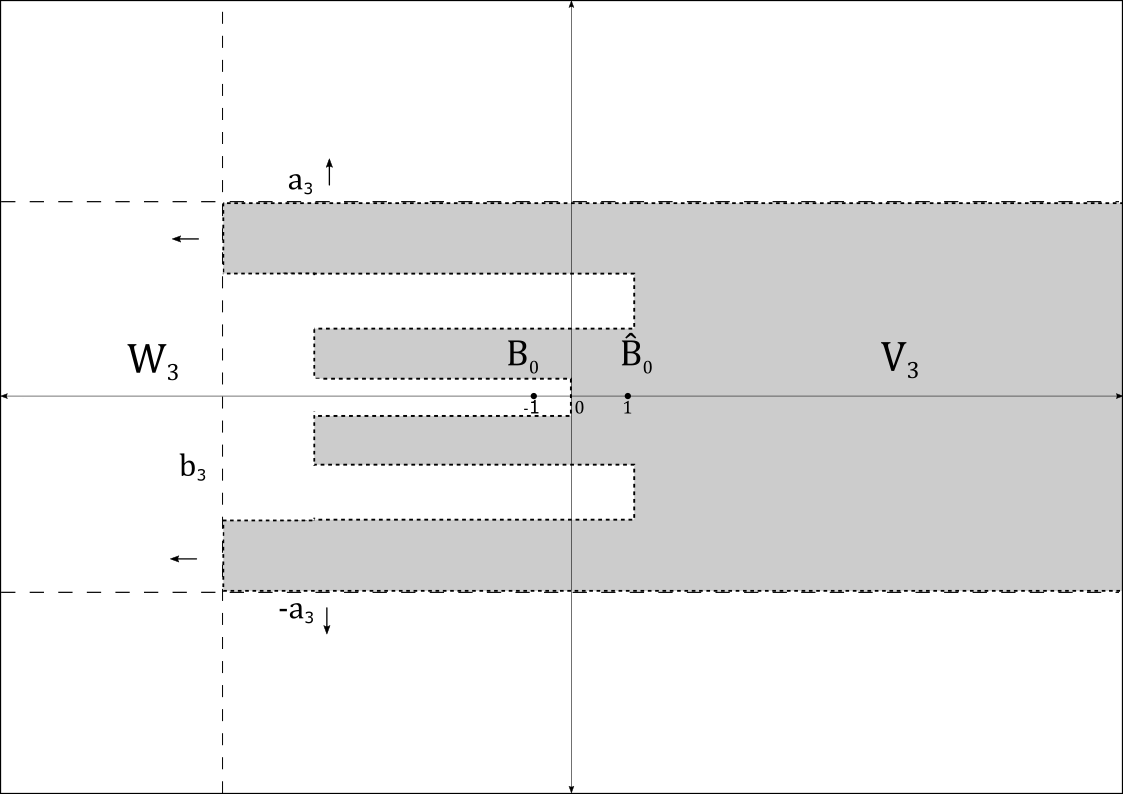}
\par\end{centering}
\begin{centering}
\includegraphics[scale=0.2]{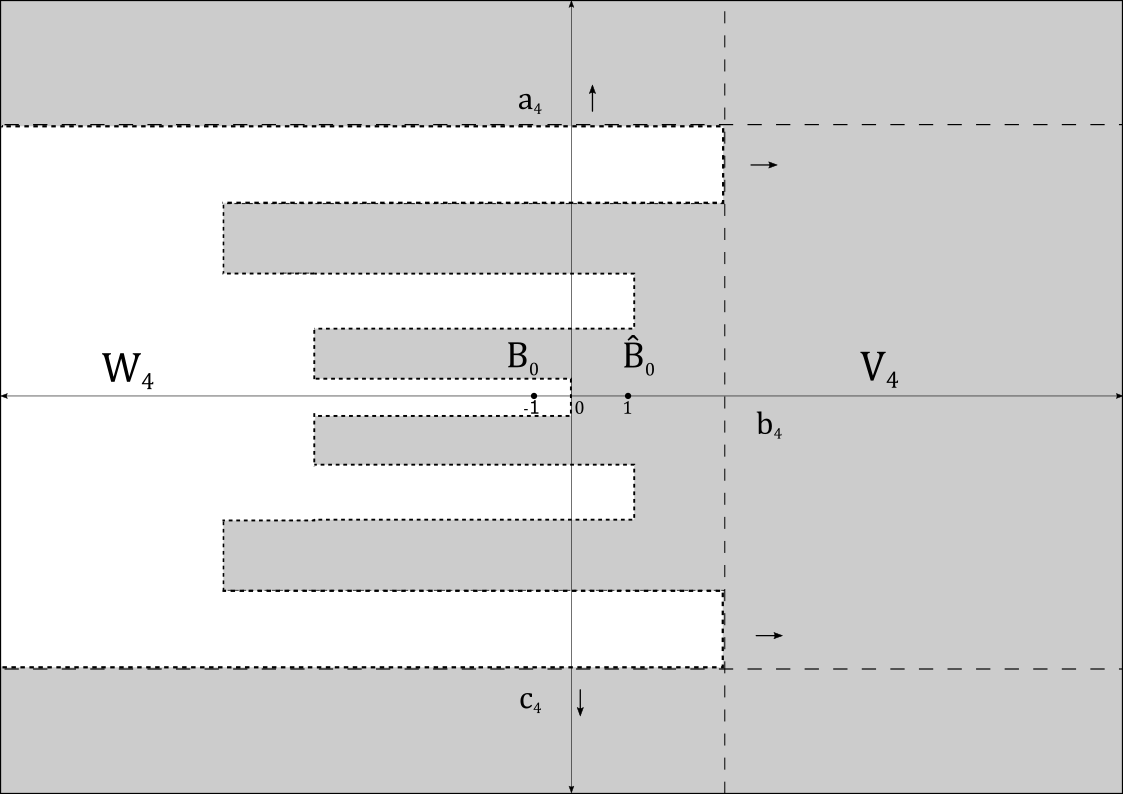}\includegraphics[scale=0.2]{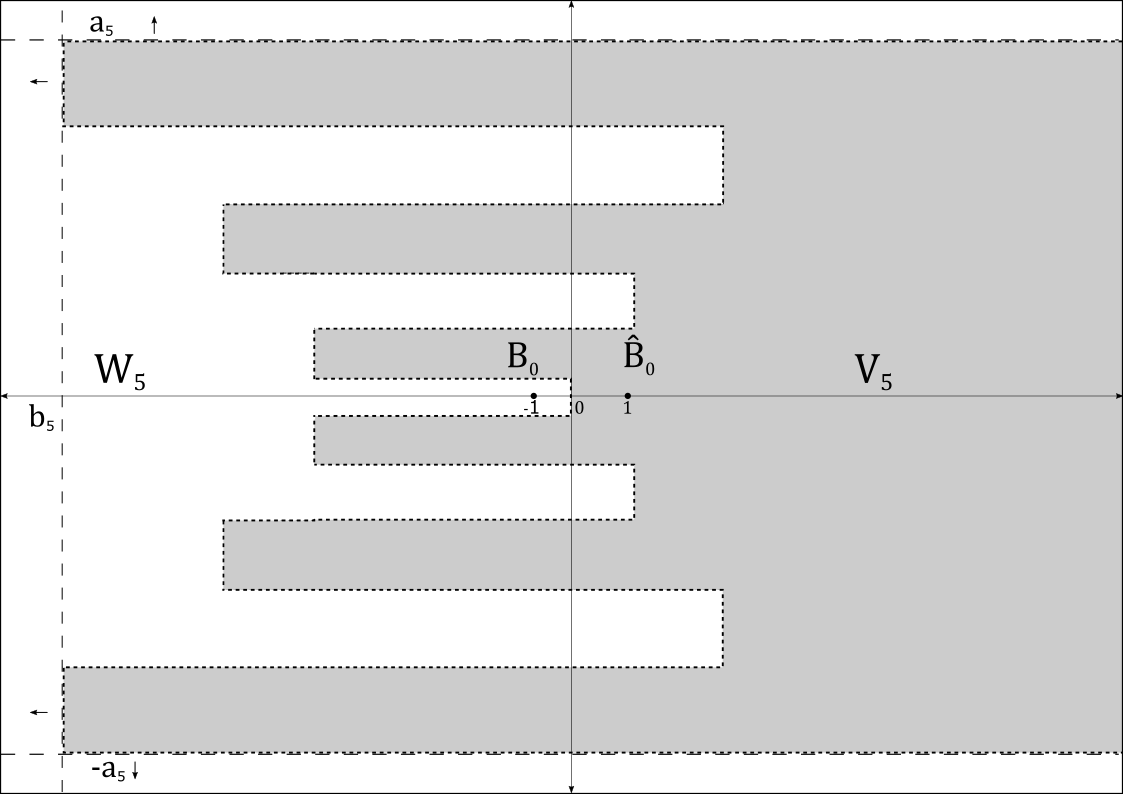}
\par\end{centering}
\caption{\label{fig:345iterationCE1}$3^{rd},$$4^{th}$, and $5^{th}$ iterations
of domain}

\end{figure}

Notice $V_{2n+1}\subset V_{2n+3}$ and $W_{2n+2}\subset W_{2n+4}$
for all $n\in\mathbb{N}_{0}$. Thus we define the required domains
\[
V_{\infty}=\bigcup_{n=0}^{\infty}V_{2n+1}\qquad W_{\infty}=\bigcup_{n=1}W_{2n+2}.
\]

By construction, the $a_n$'s go to $+\infty$, so $V_{\infty}$ and $W_{\infty}$ are both Jordan* domains. Since $(V_{2n+1})_{n=0}^{\infty}$
is an increasing sequence of domains, by Theorem \ref{thm:combthm} we have 

\[
\mathbb{E}[\tau(V_{2n+1})^{\frac{1}{4}}]\rightarrow\mathbb{E}[\tau(V_{\infty})^{\frac{1}{4}}]=\infty.
\]

The same is true of $W_{\infty}$ and thus in the limit we have a
domain and its complement who both have infinite $1/4^{th}$ moment of Brownian exit time,
as required.

\end{proof}

\subsection{Hyperbolic Geometry and Hardy Numbers}

We now construct a second counterexample to the conjecture in which the domain and
its complement both have finite $p^{th}$ moment of the exit time
for any $p>0$. To construct this second counterexample we must recall theory
from hyperbolic geometry (\cite{article}, \cite{karafyllia2021hardy}). We first define hyperbolic distance on the unit disc, which is induced
by the hyperbolic density.

\begin{defn}
(Hyperbolic density)

The hyperbolic density on the unit disc $\mathbb{D}$ is defined as

\[
\lambda_{\mathbb{D}}(z)|dz|=\frac{2|dz|}{1-|z|^{2}}
\]
\end{defn}

The definition of hyperbolic density is motivated by the following
isometric property

\[
\lambda_{\mathbb{D}}(\phi(z))|\phi'(z)|=\lambda_{\mathbb{D}}(z)
\]

for a conformal automorphism $\phi$. The hyperbolic density induces
hyperbolic length and hyperbolic distance in the following way.

\begin{defn}
(Hyperbolic length)

For any two points, $z,w\in\mathbb{D}$, we define the hyperbolic
length $\ell_{\mathbb{D}}$ as

\[
\ell_{\mathbb{D}}(\gamma)=\int_{\gamma}\lambda_{\mathbb{D}}(z)|dz|
\]
\end{defn}

As such, we define the hyperbolic distance as follows:
\begin{defn}
(Hyperbolic distance)

The hyperbolic distance on the unit disc $\mathbb{D}$, for $z,w\in\mathbb{D}$
is defined by

\[
d_{\mathbb{D}}(z,w)=\inf_{\gamma}\ell_{\mathbb{D}}(\gamma)
\]

More generally, for any simply connected domain $U\neq\mathbb{C}$,
any conformal $\phi:\mathbb{D}\rightarrow U$, and $z,w\in U$,

\[
d_{U}(z,w)=d_{\mathbb{D}}(\phi^{-1}(z),\phi^{-1}(w))
\]

and for any $E\subset U$,

\[
d_{U}(z,E)=\inf\{d_{U}(z,w):w\in E\}.
\]
\end{defn}

We may also a define a related quantity used to estimate the hyperbolic
distance.
\begin{defn}
(Quasi-hyperbolic distance)

For any two points, $z,w$, in a simply connected domain $U$, the
quasi-hyperbolic distance is defined as

\[
\delta_{U}(z,w)=\inf_{\gamma:z\rightarrow w}\int_{\gamma}\frac{|ds|}{|s-\partial U|}.
\]

Where $|.|$ is the Euclidean norm and the infinimum is taken over all
paths $\gamma$ which connect $z$ to $w$. 
\end{defn}

The quasi-hyperbolic distance estimates the hyperbolic distance by
the following result, which is a consequence of the celebrated Kobe$-\frac{1}{4}$
theorem \cite{carleson2013complex}.

\begin{thm}
$\frac{1}{2}\delta_{U}\le d_{U}\le2\delta_{U}$.
\end{thm}

We may apply this theory to the
theory of Hardy domains to construct the second counterexample using
the following theorem given in \cite{karafyllia2019hardy}.
\begin{thm}
\label{thm:bddisttobd}Let $U$ be a simply connected domain, $a\in U$,
$W_R=\{|z-a|=R\}$, and $F_{R}= W_R \cap U$ for $R>0$, then

\[
H(U)=\liminf_{R\rightarrow\infty}\frac{d_{U}(a,F_{R})}{\ln(R)}.
\]
\end{thm}

\begin{rem}
In \cite{karafyllia2019hardy}, $F_{R}$ was defined as $R\partial\mathbb{D}\cap U$, but here we require the natural extension to circles centered at arbitrary values $a$ in $U$. 
\end{rem}

Finally we will implement the following lemma.
\begin{lem}
If $U$ is simply connected and $|z-\partial U|<K$ for all $z\in U$
for some $K\in\mathbb{R^{+}}$, then $H(U)=\infty$.
\end{lem}

\begin{proof}
By definition and assumption we have
\[
\delta_{U}(a,z)=\inf_{\gamma:a\rightarrow z}\int_{\gamma}\frac{|ds|}{|s-\partial U|}>\inf_{\gamma:a\rightarrow z}\int_{\gamma}\frac{|ds|}{K}>\frac{|z-a|}{K}.
\]

From (\ref{thm:bddisttobd}), 

\[
d_{U}(a,z)\ge\frac{1}{2}\delta_{U}(a,z)>\frac{|z-a|}{2K}.
\]
Hence 

\[
H(U)=\liminf_{R\rightarrow\infty}\frac{d_{U}(a,F_{R})}{\ln(R)}\ge\liminf_{R\rightarrow\infty}\frac{R}{2K\text{\ensuremath{\ln(R)}}}=\infty.
\]
\end{proof}
\begin{example}
We will construct a domain and its complement which will both have
with finite $p^{th}$ moment for any $p$ for the exit time of a Brownian
motion exiting the domains. 

Let $\gamma_{1}=\{t\ge0:te^{it}\}$ shown in (\ref{fig:counterexample2})
in orange and $\gamma_{2}=\{t:te^{i(t-\pi)}\}$ shown in blue. Concatenating
these curves we split the complex plane into $U$ and $\mathbb{C}\backslash U$.
Any point $a$ in $U$ or $b$ in $\mathbb{C}\backslash U$ will be
within $\pi$ of $\gamma_{1}$ or $\gamma_{2}$. Hence by (\ref{thm:bddisttobd}),
$H(U)=\infty$ and $H(\mathbb{C}\backslash U)=\infty$ meaning in
turn both have exit times with finite $p^{th}$ moment for any $p$. 

\begin{figure}[H]
\begin{centering}
\includegraphics[scale=0.35]{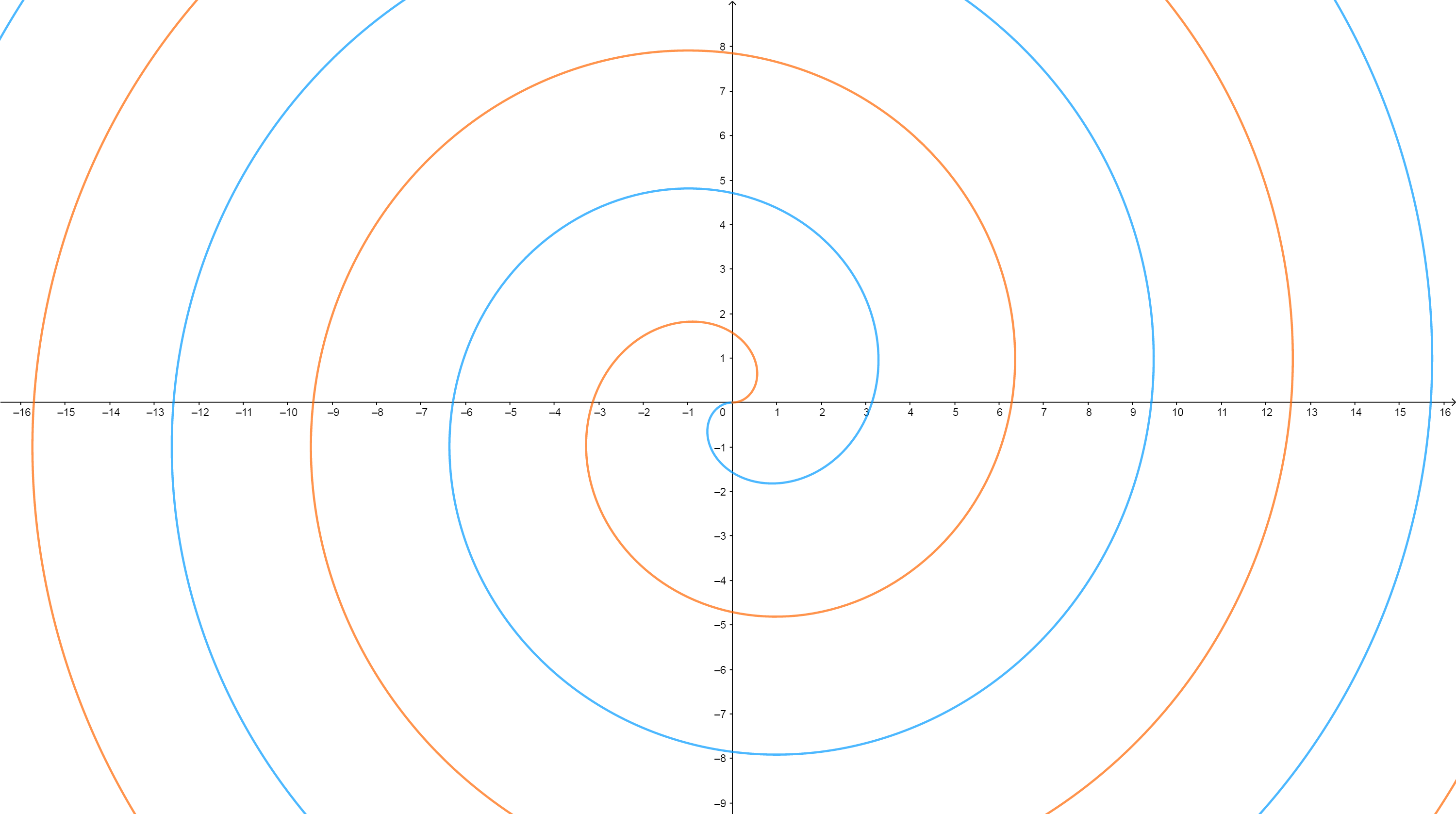}
\par\end{centering}
\centering{}\caption{\label{fig:counterexample2}Domain and complement with infinite Hardy
numbers}
\end{figure}
\end{example}

\section{Characterisations of the Cauchy Distribution}

In this short final section, we examine a method of deducing identities related to planar Brownian motion by applying the optional stopping theorem to complex-valued martingales. This method has been used previously in \cite{chin2019some} to deduce the exit distributions of Brownian motion from a half-plane and from a strip, as well as the identity

\[
\mathbb{E}[e^{i\lambda\frac{2}{\pi}\ln|C_{1}|}]=\frac{1}{\cosh\lambda},
\]

with $C_1$ a standard Cauchy, which was proved in \cite{bourg} by other methods (however, unbeknownst to the authors of \cite{chin2019some} at the time, the arguments for the strip and half-plane were already present in Exercises 2.18 and 2.19 in \cite{morters2010brownian}). \\

We will focus the method on two identities of the Cauchy distribution  given
in \cite{okamura2020characterizations} and proved via the
residue theorem. Here alternate probabilistic proofs are offered.
\begin{prop}
Let C be distributed $\textnormal{Cauchy}(a,b)$ and $\gamma=a+bi\in\mathbb{C}$.
If $\alpha\in\mathbb{C}$ and $\Im(\alpha)>0$ then $\mathbb{E}[\frac{C-\alpha}{C-\bar{\alpha}}]=\frac{\gamma-\alpha}{\gamma-\bar{\alpha}}$.
\end{prop}

\begin{proof}
Start a Brownian motion at $\gamma$ in the upper half plane and stop
it at $T(\mathbb{H})$. As is well-known (see \cite{chin2019some}), $B_{T(\mathbb{H})}\sim\textnormal{Cauchy}(a,b)$
so $B_{T(\mathbb{H})}\overset{d}{=}C$. We can transform this distribution
by $\phi(z)=\frac{z-\alpha}{z-\bar{\alpha}}$ which maps $\mathbb{H}$
to $\mathbb{D}$ and $\alpha$ to $0$. Since $\phi$ maps to $\mathbb{D}$,
it is bounded and since it is analytic, $\phi(B_{t})$ is a martingale. Thus we may utilise
the optional stopping theorem.

\[
\mathbb{E}[\frac{C-\alpha}{C-\bar{\alpha}}]=\mathbb{E}[\frac{B_{\tau}-\alpha}{B_{\tau}-\bar{\alpha}}]=\mathbb{E}[\frac{B_{0}-\alpha}{B_{0}-\bar{\alpha}}]=\frac{\gamma-\alpha}{\gamma-\bar{\alpha}}.
\]

\end{proof}

In proving the second Cauchy distribution identity, the following
results from Burkholder \cite{BURKHOLDER1977182} will be utilised.

\begin{thm}
\label{thm:SCBD}If $0<p<\infty$, $B_{t}$ is a planar Brownian motion
and $\tau$ is a stopping time such that $\mathbb{E}[\textnormal{Log}(1+\tau)]<\infty$,
then

\[
\mathbb{E}(\sup_{0\le t<\infty}|B_{\tau\wedge t}|)^{p}\le c_{p}\mathbb{E}|B_{\tau}|^{p}.
\]
\end{thm}


{\it Remark:} Clearly, if $\mathbb{E}[\tau^{p}]<\infty$ for some $p>0$,
then $\mathbb{E}[\textnormal{\ensuremath{\ln}}(1+\tau)]<\infty$ and Theorem \ref{thm:SCBD} may be applied. \\

We will also use the following convergence theorem given in \cite{knight1984kai}.

\begin{lem}
\label{lem:OSTalt}If $X=(X_{n},\mathcal{F}_{n})$ is a discrete time
submartingale and, for some $p>1$,

\[
\sup_{n}\mathbb{E}|X_{n}|^{p}<\infty
\]

then there is an integrable random variable $X_{\infty}$ such that 

\[
X_{n}\rightarrow X_{\infty}\ a.s.\qquad X_{n}\overset{L^{1}}{\rightarrow}X_{\infty}.
\]
 
\end{lem}

We thus show a probabilistic proof for the following Cauchy distribution
identity.

\begin{prop}
Let C be distributed $\textnormal{Cauchy}(a,b)$ and $\gamma=a+bi\in\mathbb{C}$.
Let $\phi(z)=z^{\alpha}$ be defined for $z\in\mathbb{C}$ and $\alpha\in(0,1)$,
using a principal branch of the logarithm; that is, $z^{\alpha}=e^{\alpha\textnormal{Log}(z)}$,
where $\textnormal{Log}(z)=\text{\ensuremath{\ln|z|+i\textup{Arg}(z)}}$
and $-\pi<\textnormal{\textup{Arg}}(z)\le\pi$. Then
\[
\mathbb{E}[C^{\alpha}]=\gamma^{\alpha}.
\]
\end{prop}

\begin{proof}
Start a Brownian motion at $\gamma\in\mathbb{H}$ and stop it at $T(\mathbb{H})$.
$\phi$ is defined to be analytic, thus $B_{t}^{\alpha}=M_{t}$ is
a martingale. Since $\alpha\in(0,1)$, $\phi(\mathbb{H})\subset\mathbb{H}$
and is a wedge of size $\alpha\pi$. Define $\tau=T(\phi(\mathbb{H}))$.
Thus, using (\ref{thm:simplyconnecthardy}) for $p<\frac{1}{\alpha}$,
$\mathbb{E}(\tau^{\frac{p}{2}})<+\infty$, and so
$\mathbb{E}[\textnormal{Log}(1+\tau)]<\infty$. \\

Thus, by (\ref{thm:SCBD}), for any $0<q<\infty$

\[
\mathbb{E}_{x}[(\sup_{0\le t<\infty}|B_{\tau\wedge t}|)^{q}]\le c_{q}\mathbb{E}|B_{\tau}|^{q}.
\]

We define $1<\beta<\frac{1}{\alpha}$ so that $\alpha\beta<1$. Setting
$q=\alpha\beta$ we have, for all $t$,

\begin{align*}
\sup_{0\le t<\infty}\mathbb{E}_{x}[|M_{\tau\wedge t}|{}^{\beta}] & \le\mathbb{E}_{x}[(\sup_{0\le t<\infty}|M_{\tau\wedge t}|)^{\beta}]\\
 & \le c_{p}\mathbb{E}|M_{\tau}|^{\beta}=D<\infty,
\end{align*}

where $D=\mathbb{E}|C|^{\alpha\beta}$ is a finite constant since
the Cauchy distribution has finite absolute fractional moments. By discretising
we have $\sup_{n}\mathbb{E}_{x}[|M_{\tau\wedge n}|{}^{\beta}]<\infty$
and so by (\ref{lem:OSTalt}), as $n\rightarrow\infty$, $M_{\tau\wedge n}\rightarrow M_{\tau}$
in $L^{1}$ and almost surely and thus in turn $\mathbb{E}[M_{\tau\wedge n}]\rightarrow\mathbb{E}[M_{\tau}]$.
However since $\tau\wedge n$ is bounded almost surely, we can apply
the optional stopping theorem to the stopped martingale $M_{\tau\wedge n}$
hence achieving the result 
\[
\mathbb{E}[M_{\tau\wedge n}]=\mathbb{E}[M_{0}]=\mathbb{E}[M_{\tau}]
.\]

Hence we arrive at the result

\[
\mathbb{E}[C^{\alpha}]=\mathbb{E}[B_{\tau}^{\alpha}]=\mathbb{E}[B_{0}^{\alpha}]=\gamma^{\alpha}.
\]
\end{proof}



\bibliographystyle{plain}
\bibliography{DRAFT}

\end{document}